\documentclass[envcountsame]{llncs}

\usepackage{amsmath,amssymb}
\usepackage{stmaryrd,enumerate,mathpartir}
\usepackage{graphicx,graphbox,hyperref}
\usepackage{tikz-cd}
\usepackage{cancel}

\newcommand{\C}{\mathbb{C}}
\newcommand{\X}{\mathbb{X}}
\newcommand{\Y}{\mathbb{Y}}

\newcommand{\REG}{\mathsf{REG}}
\newcommand{\EX}{\mathsf{EX}}
\newcommand{\DEF}{\mathsf{DEF}}
\newcommand{\RAT}{\mathsf{RAT}}

\newcommand{\SET}{\mathsf{Set}}

\newcommand{\reglex}{\mathsf{reg/lex}}
\newcommand{\exreg}{\mathsf{ex/reg}}
\newcommand{\eff}{\mathsf{eff}}
\newcommand{\tab}{\mathsf{tab}}
\newcommand{\op}{\mathsf{op}}
\newcommand{\Split}{\mathsf{Split}}
\newcommand{\per}{\mathsf{per}}
\newcommand{\cor}{\mathsf{cor}}
\newcommand{\eq}{\mathsf{eq}}
\newcommand{\Map}{\mathsf{Map}}
\newcommand{\Rel}{\mathsf{Rel}}

\makeatletter
\providecommand{\leftsquigarrow}{%
  \mathrel{\mathpalette\reflect@squig\relax}%
}
\newcommand{\reflect@squig}[2]{%
  \reflectbox{$\m@th#1\rightsquigarrow$}%
}
\makeatother

\author{Chad Nester \thanks{This research was supported by the ESF funded Estonian IT Academy research measure (project 2014-2020.4.05.19-0001).}}
\institute{Tallinn University of Technology, Tallinn, Estonia}
\title{A Variety Theorem for Relational Universal Algebra}

\begin{document}

\maketitle


\begin{abstract}
  We consider an analogue of universal algebra in which generating symbols are interpreted as relations. We prove a variety theorem for these relational algebraic theories, in which we find that their categories of models are precisely the definable categories. The syntax of our relational algebraic theories is string-diagrammatic, and can be seen as an extension of the usual term syntax for algebraic theories. 
\end{abstract}


\section{Introduction}

Universal algebra is the study of what is common to algebraic structures, such as groups and rings, by algebraic means. The central idea of universal algebra is that of a \emph{theory}, which is a syntactic description of some class of structures in terms of generating symbols and equations involving them. A \emph{model} of a theory is then a set equipped with a function for each generating symbol in a way that satisfies the equations. There is a further notion of \emph{model morphism}, and together the models and model morphisms of a given theory form a category. These categories of models are called \emph{varieties}. Much of classical algebra can be understood as the study of specific varieties. For example, group theory is the study of the variety of groups, which arises from the theory of groups in the manner outlined above. 

A given variety will in general arise as the models of more than one theory. A natural question to ask, then, is when two theories present the same variety. To obtain a satisfying answer to this question it is helpful to adopt a more abstract perspective. Theories become categories with finite products, models become functors, and model morphisms become natural transformations. Our reward for this shift in perspective is the following answer to our question: two theories present equivalent varieties in case they have equivalent idempotent splitting completions. Thus, from a certain point of view universal algebra is the study of categories with finite products.

This point of view has developed into \emph{categorical} universal algebra. For any sort of categorical structure we can treat categories with that structure as theories, functors that preserve it as models, and natural transformations thereof as model morphisms. The aim is then to figure out what sort of categories arise as models and model morphisms of this kind -- that is, to determine the appropriate notion of variety. For example, if we take categories with finite limits to be our theories, then varieties correspond to locally finitely presentable categories ~\cite{Adamek1994}.

The familiar syntax of classical algebra -- consisting of \emph{terms} built out of variables by application of the generating symbols -- is inextricably bound to finite product structure. In leaving finite products behind for more richly-structured settings, categorical universal algebra also leaves behind much of the syntactic elegance of its classical counterpart. While methods of specifying various sorts of theory (categories with structure) exist, these are often cumbersome, lacking the intuitive flavour of classical universal algebra. 

The present paper concerns an analogue of classical universal algebra in which the generating symbols are understood as \emph{relations} instead of functions. The role of classical terms is instead played by string diagrams, and categories with finite products become cartesian bicategories of relations in the sense of ~\cite{CarboniWalters1987} -- an idea that first appears in ~\cite{Bonchi2017}. This allows us to present relational algebraic theories in terms of generators and equations, in the style of classical universal algebra. In fact, this approach to syntax for relational theories extends the classical syntax for algebraic theories, which admits a similar diagrammatic presentation.

Our development is best understood in the context of recent work on partial algebraic theories ~\cite{partial-theories}, in which the string-diagrammatic syntax for algebraic theories is modified to capture partial functions. This modification of the basic syntax coincides with an increase in the expressive power of the framework, corresponding roughly to the equalizer completion of a category with finite products ~\cite{Bunge1995}. The move to relational algebraic theories involves a further modification of the string-diagrammatic syntax, corresponding roughly to the regular completion of a category with finite limits ~\cite{CarboniVitale1998}. Put another way, in ~\cite{partial-theories} the (string-diagrammatic) syntax for algebraic theories is extended to express a certain kind of equality, and the resulting terms denote partial functions. In this paper, we further extend the string-diagrammatic syntax to express existential quantification, and the resulting terms denote relations 

\emph{Contributions}. The central contribution of this paper is a variety theorem characterizing the categories that arise as the models and model morphisms of some relational algebraic theory (Theorem \ref{thm:relvariety}). Specifically, we will see that these are precisely the \emph{definable} categories of ~\cite{Kuber2018}. As a consequence we obtain that two relational algebraic theories present the same definable category if and only if splitting the partial equivalence relations in each yields equivalent categories (Theorem \ref{thm:persplit}). We illustrate the use of our framework with a number of examples, including the theory of regular semigroups ~\cite{Green1951} and the theory of effectoids ~\cite{Tate2013}. Lemma \ref{lem:easybicatrel} is also novel, and we consider it to be a minor contribution

\emph{Related Work}. The study of universal algebra began with the work of Birkhoff ~\cite{Birkhoff1935}. A few decades later, Lawvere introduced the categorical perspective in his doctoral thesis ~\cite{Lawvere1963}. A modern account of universal algebra from the categorical perspective is ~\cite{AlgebraicTheories2010}. A highlight of this account is the variety theorem for algebraic theories ~\cite{Adamek2003}, which our variety theorem for relational algebraic theories is explicitly modelled on. An important result in categorical algebra is Gabriel-Ulmer duality ~\cite{Gabriel1972}, which tells us that if we consider categories with finite limits as our notion of algebraic theory, then the corresponding notion of variety is that of a locally finitely presentable category ~\cite{Adamek1994}. Our development relies on the related notion of a definable category ~\cite{Kuber2018,LackTendas20}, which recently arose in the development of an analogue of Gabriel-Ulmer duality for regular categories. 

We use cartesian bicategories of relations ~\cite{CarboniWalters1987} as our notion of relational algebraic theory. Our development relies on several results from the theory of allegories ~\cite{Freyd90}, in which cartesian bicategories of relations coincide with the notion of a unitary pre-tabular allegory. We also make use of the theory of regular and exact completions ~\cite{CarboniVitale1998}. Of course, all of this relies on the theory of regular and exact categories ~\cite{Barr1971}. The idea of using string diagrams as terms in more general notions of algebraic theories is relatively recent, and relies on the work of Fox ~\cite{Fox76}. The present paper can be considered a generalisation of recent work on partial theories ~\cite{partial-theories} to include relations. The idea to treat cartesian bicategories of relations as theories with models in the category of sets and relations originally appeared in ~\cite{Bonchi2017}, although no variety theorem is provided therein.

\emph{Organization and Prerequisites}. In Section \ref{sec:algrelations} we introduce categories of abstract relations. In Section \ref{sec:examples} we give the definition of a relational algebraic theory, and provide a number of examples. Section \ref{sec:varietytheorem} contains the proof of the variety theorem. We assume familiarity with category theory, including regular categories ~\cite{Barr1971}, string diagrams for monoidal categories ~\cite{Joy91} and their connection to algebraic theories ~\cite{AlgebraicTheories2010}, and some 2-category theory ~\cite{Kelly1974}. We will behave as though all monoidal categories are \emph{strict} monoidal categories, justifying this behaviour in the usual way by appealing to the coherence theorem for monoidal categories ~\cite{Mac71}.


\section{The Algebra of Relations}\label{sec:algrelations}

In the context of algebraic theories, finite product structure serves as an algebra of functions. In this section, we consider an analogous algebra of relations. There are two perspectives from which to consider this algebra of relations: As internal relations in a regular category, or through cartesian bicategories of relations. The two perspectives are very closely related, and we require both: it is through regular categories that our development connects to the wider literature on categorical algebra, but our syntax for relational theories will be the string-diagrammatic syntax for cartesian bicategories of relations. 

To begin, we recall the category $\Rel$ of sets and relations, which will serve as the universe of models for relational theories in the same way that the category $\SET$ of sets and functions is the universe of models for classical algebraic theories. 
\begin{definition}\label{def:relations}
The category $\Rel$ has sets as objects, with arrows $f : X \to Y$ given by binary relations $f \subseteq X \times Y$. The composite of arrows $f : X \to Y$, $g : Y \to Z$ is defined by $fg = \{ (x,z) \mid \exists y \in Y . (x,y) \in f \wedge (y,z) \in g \}$, and the identity relation on $X$ is $\{ (x,x) \mid x \in X\}$. 
\end{definition}


\subsection{Categories of Internal Relations}

In any regular category we can construct an abstract analogue of Definition \ref{def:relations}. Instead of sub\emph{sets} $R \subseteq A \times B$, we represent relations as sub\emph{objects} $R \rightarrowtail A \times B$. This approach to categorifying the theory of relations has a relatively long history ~\cite{Freyd90}, and integrates well with standard categorical logic due to the ubiquity of regular categories there. 

\begin{definition}
  Let $\C$ be a regular category. The associated category of \emph{internal relations}, $\Rel(\C)$, is defined as follows:
  \begin{enumerate}[]
  \item \textbf{objects} are objects of $\C$
  \item \textbf{arrows} $r : A \to B$ are jointly monic spans $r = \langle f,g \rangle : R \rightarrowtail A \times B$ modulo equivalence as subobjects of $A \times B$. That is, $r : R \rightarrowtail A \times B$ and $r' : R' \rightarrowtail A \times B$ are equivalent (and thus define the same arrow of $\Rel(\C)$) in case there exists an isomorphism $\alpha : R \to R'$ such that $\alpha r' = r$. 
  \item \textbf{composition} of two arrows $r : A \to B$ and $s : B \to C$ given respectively by $\langle f,g \rangle : R \rightarrowtail A \times B$ and $\langle h,k \rangle : S \rightarrowtail B \times C$ is defined by first constructing the pullback of $h$ along $g$, pictured below on the left. This defines an arrow $\langle h'f,g'k \rangle : R \times_B S \to A \times C$. The composite $rs : A \to C$ is defined to be the monic part of the image factorization of this arrow, pictured below on the right.
    \begin{mathpar}
    \begin{tikzcd}
      R \times_B S \ar[dr,phantom,"\lrcorner" very near start] \ar[d,"h'"'] \ar[r,"g'"] & S \ar[d,"h"] \\
      R \ar[r,"g"'] & B
    \end{tikzcd}
        
    \begin{tikzcd}
      R \times_B S \ar[rd,twoheadrightarrow] \ar[rr,"{\langle h'f,g'k \rangle}"] && A \times C \\
      & RS \ar[ur,rightarrowtail,"rs"'] 
    \end{tikzcd}
    \end{mathpar}
  \item \textbf{identities} $1_A : A \to A$ are are given by diagonal maps $\Delta_A : A \rightarrowtail A \times A$. 
  \end{enumerate}
\end{definition}

\begin{example}
  $\SET$ is a regular category, and the category of internal relations in $\Rel(\SET)$ is precisely the usual category of sets and relations $\Rel$. 
\end{example}


\subsection{Cartesian Bicategories of Relations}

It is difficult to work with relations internal to a regular category directly. Routine calculations often involve complex interaction between pullbacks and image factorizations, and this quickly becomes intractable. A much more tractable setting for working with relations is provided by cartesian bicategories of relations, which admit a convenient graphical syntax.

Cartesian bicategories of relations are defined in terms of commutative special frobenius algebras, which provide the basic syntactic scaffolding of our approach: 

\begin{definition}
  Let $\X$ be a symmetric strict monoidal category. A \emph{commutative special frobenius algebra} in $\X$ is a 5-tuple $(X,\delta_X,\mu_X,\varepsilon_X,\eta_X)$, as in
  \begin{mathpar}
    {\delta_X \hspace{0.3cm} \leftrightsquigarrow \hspace{0.3cm} \includegraphics[height=1.2cm,align=c]{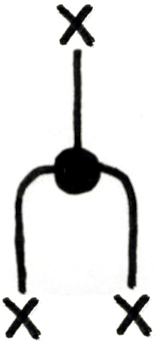}}

    {\mu_X \hspace{0.3cm} \leftrightsquigarrow \hspace{0.3cm}  \includegraphics[height=1.2cm,align=c]{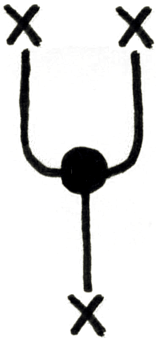}}

    {\varepsilon_X \hspace{0.3cm} \leftrightsquigarrow \hspace{0.3cm} \includegraphics[height=1.2cm,align=c]{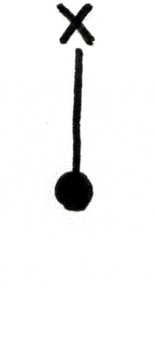}}

    {\eta_X \hspace{0.3cm} \leftrightsquigarrow \hspace{0.3cm} \includegraphics[height=1.2cm,align=c]{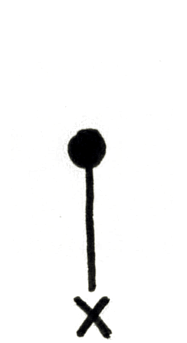}}
  \end{mathpar}
  such that
  \begin{enumerate}[(i)]
  \item $(X,\delta_X,\varepsilon_X)$ is a commutative comonoid:
    \begin{mathpar}
      \includegraphics[height=1cm,align=c]{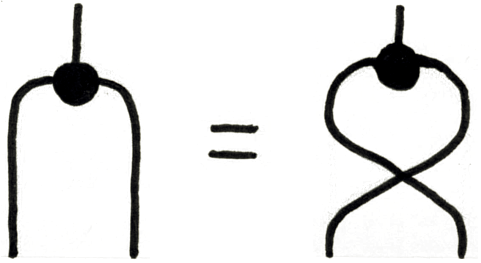}
      
      \includegraphics[height=1cm,align=c]{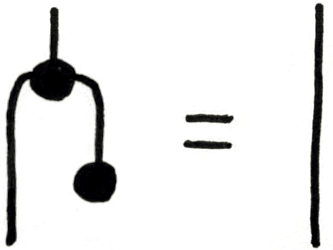}
    
      \includegraphics[height=1cm,align=c]{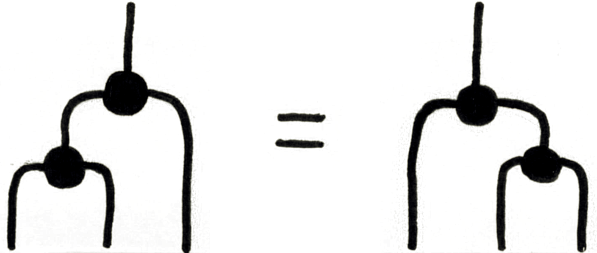}
    \end{mathpar}
  \item $(X,\mu_X,\eta_X)$ is a commutative monoid:
    \begin{mathpar}
      \includegraphics[height=1cm,align=c]{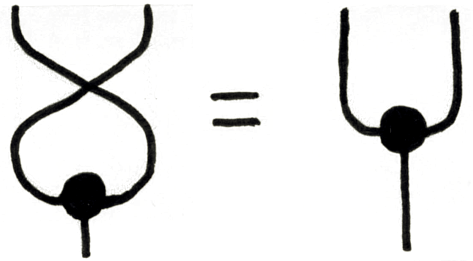}
      
      \includegraphics[height=1cm,align=c]{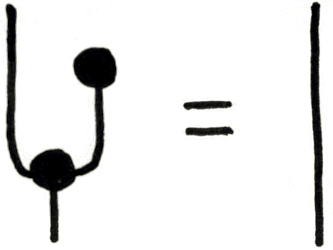}
      
      \includegraphics[height=1cm,align=c]{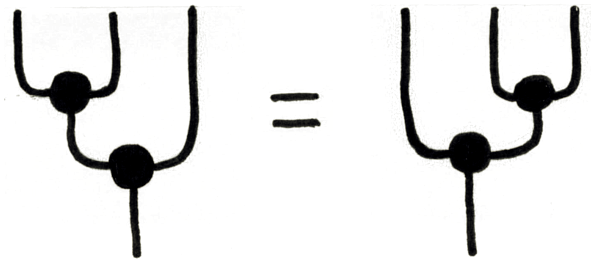}
    \end{mathpar}
  \item $\mu_X$ and $\delta_X$ satisfy the special and frobenius equations:
    \begin{mathpar}
      {\includegraphics[height=1cm,align=c]{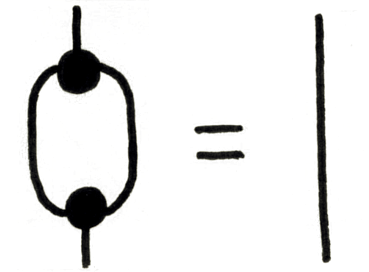}}
      
      {\includegraphics[height=1cm,align=c]{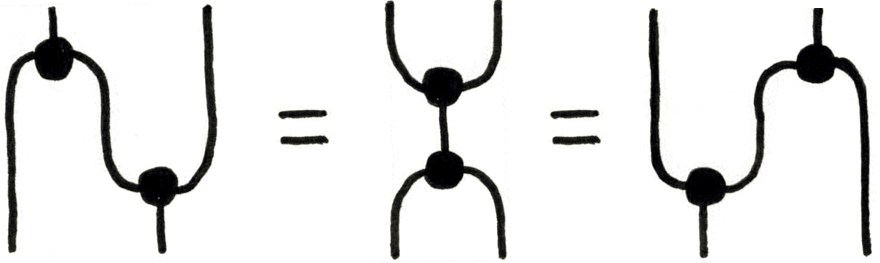}}
    \end{mathpar}
  \end{enumerate}
\end{definition}

An intermediate notion is that of a hypergraph category, in which objects are coherently equipped with commutative special frobenius algebra structure:

\begin{definition}
  A symmetric strict monoidal category $\X$ is called a \emph{hypergraph category} ~\cite{Fong2019} in case:
  \begin{enumerate}[(i)]
  \item Each object $X$ of $\X$ is equipped with a commutative special frobenius algebra.
  \item The frobenius algebra structure is coherent, i.~e., for all $X,Y$ we have:
    \begin{mathpar}
    \includegraphics[height=1.2cm,align=c]{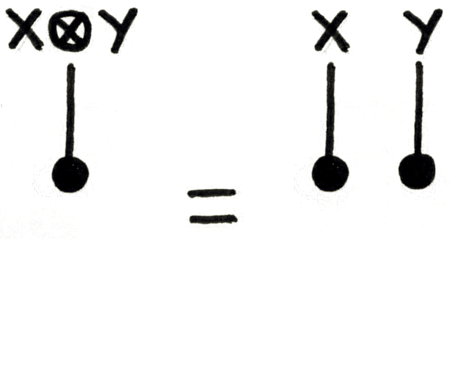}

    \includegraphics[height=1.2cm,align=c]{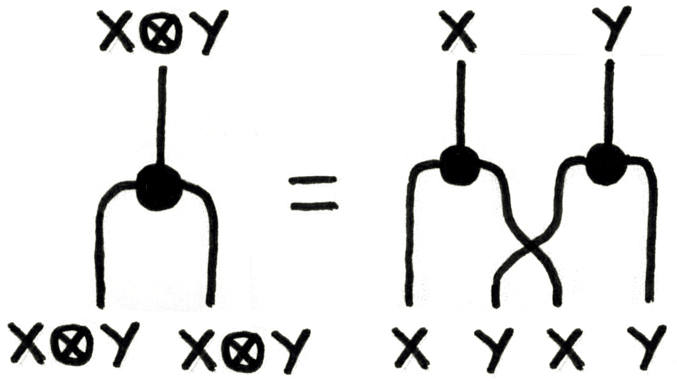}

    \includegraphics[height=1.2cm,align=c]{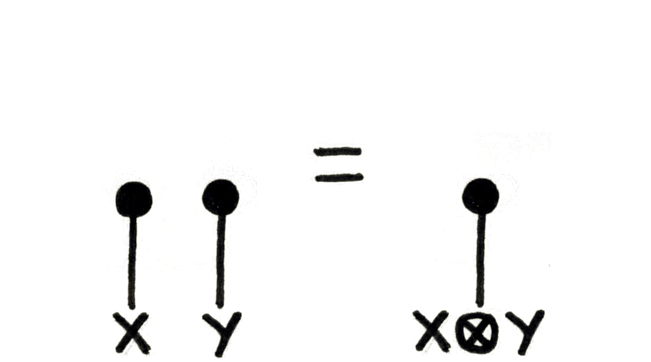}

    \includegraphics[height=1.2cm,align=c]{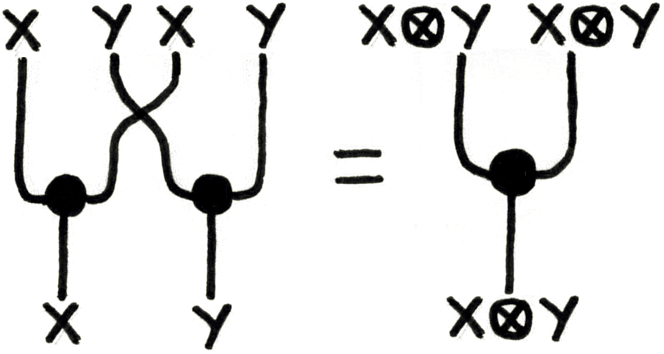}
    \end{mathpar}


      
  \end{enumerate}
\end{definition}

Now a cartesian bicategory of relations is a hypergraph category enjoying certain additional structure:

\begin{definition}\label{def:cartbicatrel}
  A \emph{cartesian bicategory of relations} ~\cite{CarboniWalters1987} is a poset-enriched hypergraph category $\X$ such that:
  \begin{enumerate}[(i)]
  \item The comonoid structure is \emph{lax natural}. That is, for all arrows $f$ of $\X$:
    \begin{mathpar}
    {\includegraphics[height=1cm,align=c]{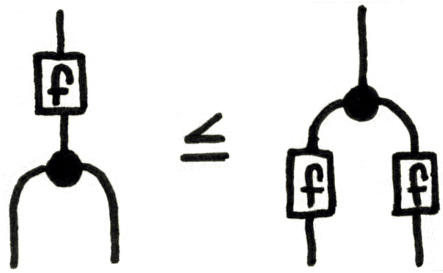}}

    {\includegraphics[height=1cm,align=c]{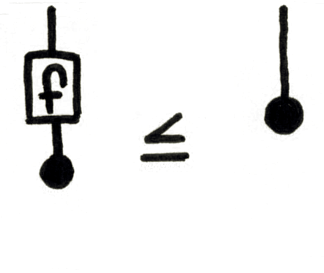}}
    \end{mathpar}
  \item Each of the frobenius algebras satisfy:
    \begin{mathpar}
    {\includegraphics[height=1cm,align=c]{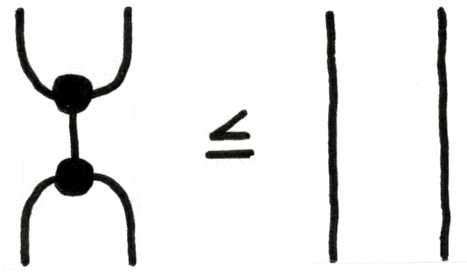}}

    {\includegraphics[height=1cm,align=c]{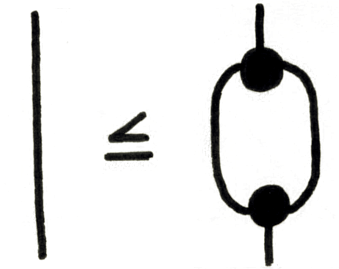}}
    
    {\includegraphics[height=1cm,align=c]{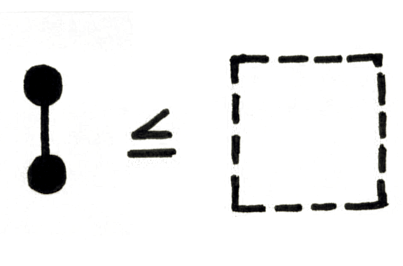}}

    {\includegraphics[height=1cm,align=c]{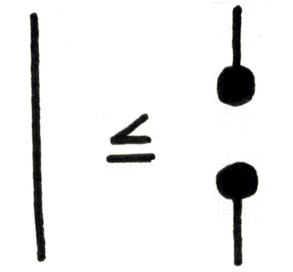}}
    \end{mathpar}
  \end{enumerate}
\end{definition}

\begin{example}
  The category $\Rel$ is a cartesian bicategory of relations with
  \begin{mathpar}
    \delta_X = \{ (x,(x,x)) \mid x \in X \}

    \mu_X = \{ ((x,x),x) \mid x \in X \}
    
    \varepsilon_X = \{ (x,*) \mid x \in X \}

    \eta_X = \{(*,x) \mid x \in X\}
  \end{mathpar}
  where $*$ is the unique element of the singleton set $I = \{*\}$. 
\end{example}

\begin{example}
  If $\C$ is a regular category then $\Rel(\C)$ is a cartesian bicategory of relations with $X \otimes Y = X \times Y$, $I = 1$, and 
  \begin{mathpar}
    \delta_X = \langle 1_X,\Delta_X \rangle : X \rightarrowtail X \times (X \times X)

    \mu_X = \langle \Delta_X,1_X \rangle : X \rightarrowtail (X \times X) \times X

    \varepsilon_X = \langle 1_X,!_X \rangle : X \rightarrowtail X \times 1

    \eta_X = \langle !_X,1_X \rangle : X \rightarrowtail 1 \times X
  \end{mathpar}
  Where $\Delta_X$ is the diagonal morphism and $!_X$ is the unique morphism into the terminal object $1$ of $\C$. 
\end{example}

Cartesian bicategories of relations admit meets of hom-sets:
\begin{lemma}[~\cite{Bonchi2017}]\label{lem:meets}
  Every cartesian bicategory of relations has meets of parallel arrows, with $f \cap g$ for $f,g : X \to Y$ defined by
  \[
  \includegraphics[height=1.2cm,align=c]{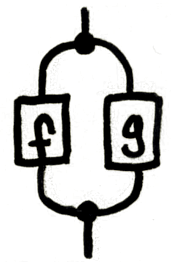}
  \]
  Further, the meet determines the poset-enrichment in that $f \leq g \Leftrightarrow f \cap g = f$.

\end{lemma}
We point out this allows for a much simpler presentation, as in:

\begin{lemma}\label{lem:easybicatrel}
  A hypergraph category $\X$ is a cartesian bicategory of relations if and only if for each arrow $f$:
  \[
  \includegraphics[height=1.2cm,align=c]{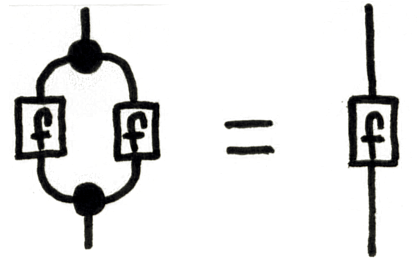}
  \]
\end{lemma}

We will require a 2-category of cartesian bicategories of relations in our development. Our notion of 1-cell is a structure-preserving functor as in:
\begin{definition}
  A \emph{morphism} of cartesian bicategories of relations $F : \X \to \Y$ is a strict monoidal functor that preserves the frobenius algebra structure:
  \begin{mathpar}
%
%
    F(\delta_X) = \delta_{FX}

    F(\mu_X) = \mu_{FX}
    
    F(\varepsilon_X) = \varepsilon_{FX}

    F(\eta_X) = \eta_{FX}
  \end{mathpar}
\end{definition}
and the correct sort of 2-cell turns out to be a \emph{lax} natural transformation:
\begin{definition}
  Let $\X,\Y$ be cartesian bicategories of relations, and let $F,G : \X \to \Y$ be morphisms thereof. Then a \emph{lax transformation} $\alpha : F \to G$ consists of an $\X_0$-indexed family of arrows $\alpha_X : F(X) \to G(X)$ such that for each arrow $f : X \to Y$ of $\X$ we have $F(f)\alpha_Y \leq \alpha_XG(f)$ in $\Y$.
\end{definition}

\begin{definition}\label{def:rat}
  Let $\RAT$ be the 2-category with cartesian bicategories of relations as 0-cells, their morphisms as 1-cells, and lax transformations as 2-cells.
\end{definition}

An important class of arrows in a cartesian bicateory of relations are the \emph{maps}, which should be thought of as those relations that happen to be functions.

\begin{definition}[Maps]
  An arrow $f : X \to Y$ in a cartesian bicategory of relations is called:
  \begin{enumerate}[(i)]
  \item \emph{simple} in case the equation below on the left holds.
  \item \emph{total} in case the equation below on the right holds.
    \begin{mathpar}
    \includegraphics[height=1cm,align=c]{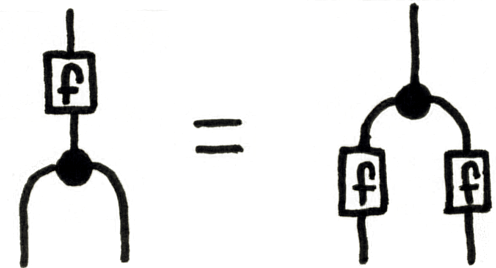}
      
    \includegraphics[height=1cm,align=c]{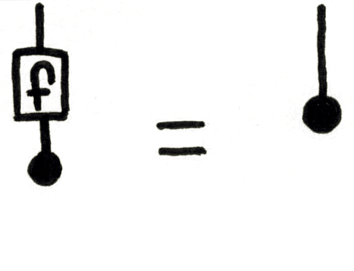}
    \end{mathpar}
  \item A \emph{map} in case it is both simple and total.
  \end{enumerate}
\end{definition}
The maps of a cartesian bicategory of relations always form a subcategory $\Map(\X)$. For example, $\Map(\Rel) \cong \SET$. More generally:


\begin{theorem}[~\cite{Freyd90}]
  For $\C$ a regular category, there is an equivalence of categories $\C \simeq \Map(\Rel(\C))$.
\end{theorem}
Remarkably, the components of lax transformations are always maps:
\begin{lemma}[~\cite{Bonchi2017}]\label{lem:totalcomponents}
  If $\X,\Y$ are cartesian bicategories of relations, $F,G : \X \to \Y$ are morphisms thereof and $\alpha : F \to G$ is a lax transformation, then each component $\alpha_X : FX \to GX$ of $\alpha$ is necessarily a map. 
\end{lemma}

\section{Relational Algebraic Theories}\label{sec:examples}

In this section we define relational algebraic theories along with the models and model morphisms, and consider a number of examples. 

\begin{definition}~\cite{Bonchi2017}
  A \emph{relational algebraic theory} is a cartesian bicategory of relations. A \emph{model} of a relational algebraic theory $\X$ is a morphism of cartesian bicategories of relations $F : \X \to \Rel$. A \emph{model morphism} $\alpha : F \to G$ is a lax transformation. 
\end{definition}

It is convenient to present relational algebraic theories somewhat informally in terms of string-diagrammatic generators and (in)equations between them, with the structure of a cartesian bicategory of relations implicitly present. A more formal account would proceed in terms of monoidal equational theories, from which the cartesian bicategory of relations giving the associated relational algebraic theory may be freely constructed ~\cite{Bonchi2017}.

\begin{example}[Sets]
  The relational algebraic theory with no generators and no equations has sets as models and functions as model morphisms (see Lemma \ref{lem:totalcomponents}), and so the associated category of models is $\SET$.
\end{example}

\begin{example}[Posets]
  Consider the relational theory with a single generator (below left) which is required to be reflexive, transitive, and antisymmetric:
  \begin{mathpar}
    \includegraphics[height=1cm,align=c]{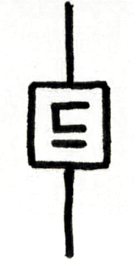}
    
  \includegraphics[height=1cm,align=c]{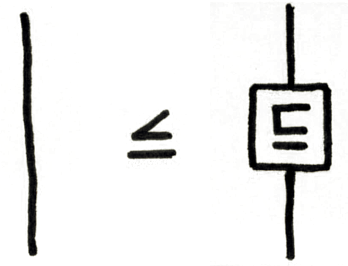}

  \includegraphics[height=1cm,align=c]{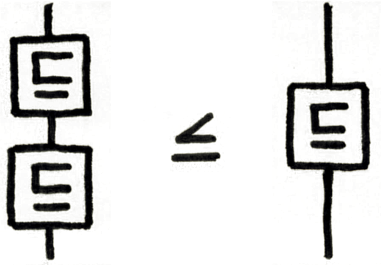}

  \includegraphics[height=1cm,align=c]{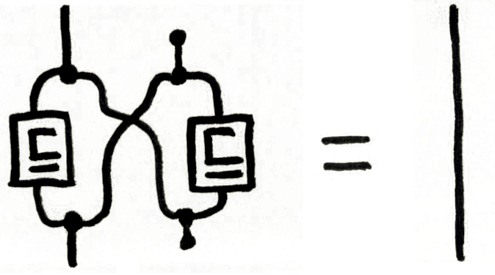}
  \end{mathpar}
  The associated category of models is the category of posets and monotone maps.
\end{example}

\begin{example}[Nonempty Sets]
  Consider the relational theory with no generating symbols and a single equation:
  \[
  \includegraphics[height=1cm,align=c]{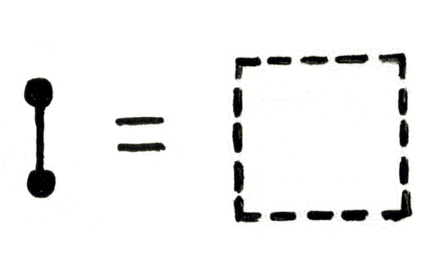}
  \]
  Models of the associated relational algebraic theory are sets $X$ such that the generating equation is satisfied in $\Rel$: 
  \[ \eta_X\varepsilon_X = \{(*,*)\} = \square_I \] 
  where $\eta_X$ and $\varepsilon_X$ are defined as in Definition \ref{def:relations}. If we calculate the relational composite, we find that:
  \begin{align*}\eta_X\varepsilon_X &= \{(*,*) \mid \exists x \in X. (*,x) \in \eta_X \wedge (x,*) \in \varepsilon_X \} = \{ (*,*) \mid \exists x \in X \} \end{align*}
  and so models are nonempty sets. The theory of nonempty sets contains no generating morphisms, and so model morphisms are simply functions. Contrast this to the category of \emph{pointed} sets, in which morphisms must preserve the point.
\end{example}

\begin{example}[Regular Semigroups]\label{ex:regsemi}
  A \emph{semigroup} is a set equipped with an associative binary operation, denoted by juxtaposition. A semigroup $S$ is \emph{regular} ~\cite{Green1951} in case
  \[ \forall a \in S . \exists x \in S. axa = a\]
  The relational theory of semigroups has a single generating symbol (below left) which is required to be simple, total, and associative:
  \begin{mathpar}
    \includegraphics[height=1cm,align=c]{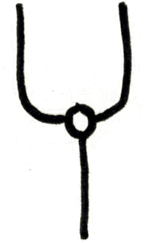}
    
    \includegraphics[height=1cm,align=c]{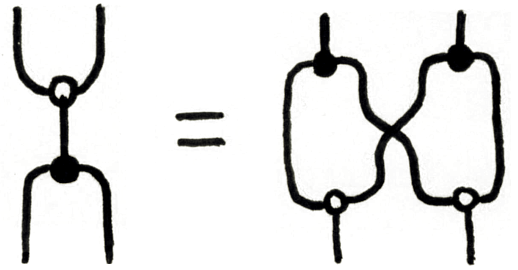}

    \includegraphics[height=1cm,align=c]{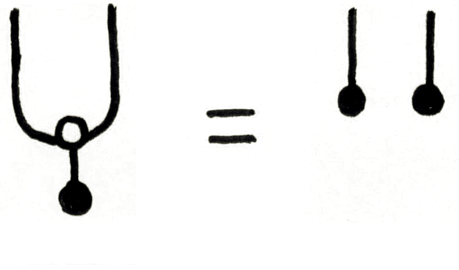}
    
    \includegraphics[height=1cm,align=c]{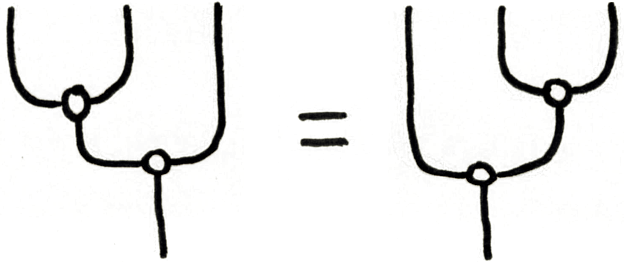}
  \end{mathpar}
  To capture the regular semigroups we include the following equation:
  \[
  \includegraphics[height=1cm,align=c]{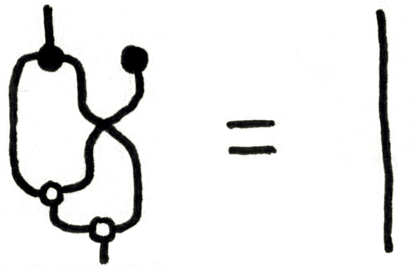}
  \]
  The associated category of models is the category of regular semigroups and semigroup homomorphisms. 
\end{example}


\begin{example}[Effectoids]\label{ex:effectoids}
  An \emph{effectoid} ~\cite{Tate2013} is a set $A$ equipped with a unary relation $\cancel{\varepsilon} \mapsto \_ \subseteq A$, a binary relation $\_ \preceq \_ \subseteq A \times A$, and a ternary relation $\_\,;\_\mapsto \_ \subseteq A \times A \times A$ satisfying:
\begin{enumerate}[]
  \item (Identity) For all $a,a' \in A$,
    \begin{mathpar} \exists x \in A. (\cancel{\varepsilon} \mapsto x) \wedge (x\,;a \mapsto a') \Leftrightarrow a \preceq a' \Leftrightarrow \exists y \in A. (\cancel{\varepsilon} \mapsto y) \wedge (a\,;y \mapsto a')
    \end{mathpar}
  \item (Associativity) For all $a,b,c,d \in A$, 
    \begin{mathpar}
      \exists x . (a\,;b \mapsto x) \wedge (x\,;c \mapsto d)  \Leftrightarrow  \exists y . (b\,;c \mapsto y) \wedge (a\,;y \mapsto d)
    \end{mathpar}
  \item (Reflexive Congruence 1) For all $a \in A$, $a \preceq a$.
  \item (Reflexive Congruence 2) For all $a,a' \in A$, $(\cancel{\varepsilon} \mapsto a) \wedge (a \preceq a') \Rightarrow (\cancel{\varepsilon} \mapsto a')$
  \item (Reflexive Congruence 3) For all $a,b,c \in A$, $\exists x . (a\,;b \mapsto x) \wedge (x \preceq c) \Rightarrow (a\,;b \preceq c)$
  \end{enumerate}

To obtain a relational theory of effectoids, we ask for three generating symbols corresponding respectively to the unary, binary, and ternary relation:
  \begin{mathpar}
    \includegraphics[height=1cm,align=c]{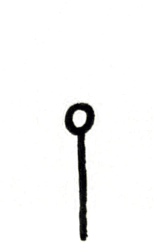}

    \includegraphics[height=1cm,align=c]{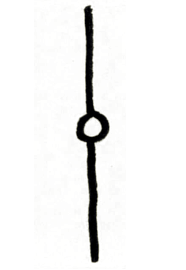}
    
    \includegraphics[height=1cm,align=c]{figs/string-diagram-circ.png}
  \end{mathpar}
  Then the identity and associativity axioms become:
  \begin{mathpar}
    \includegraphics[height=1cm,align=c]{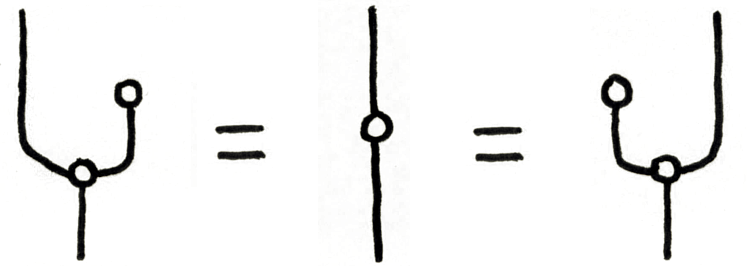}

    \includegraphics[height=1cm,align=c]{figs/associativity-axiom.png}
  \end{mathpar}
  And the reflexive congruence axioms become:
  \begin{mathpar}
    \includegraphics[height=1cm,align=c]{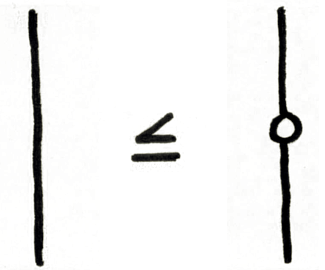}

    \includegraphics[height=1cm,align=c]{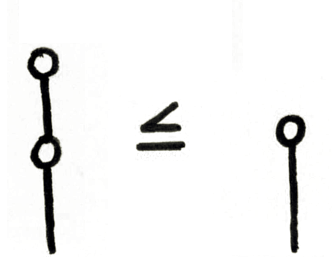}

    \includegraphics[height=1cm,align=c]{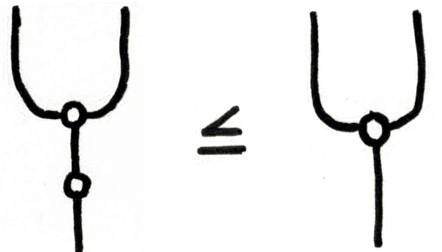}
  \end{mathpar}
  The models of this relational theory are precisely the effectoids.
\end{example}

\begin{example}[Generalized Separation Algebras]
  A \emph{generalized separation algebra} ~\cite{Jipsen18} is a partial monoid satisfying the left and right cancellativity axioms, which further satisfies the conjugation axiom:
  \[ \forall x,y. (\exists z.x \circ z = y) \Leftrightarrow (\exists w.w \circ x = y) \]
  To capture generalized separation algebras as a relational algebraic theory, we require two generating symbols in the generating monoidal equational theory, corresponding to the monoid operation and the unit:
  \begin{mathpar}
    \includegraphics[height=1cm,align=c]{figs/string-diagram-circ.png}

    \includegraphics[height=1cm,align=c]{figs/string-diagram-unit.png}
  \end{mathpar}
  Both are required to be simple, and the unit is required to be total:
  \begin{mathpar}
    \includegraphics[height=1cm,align=c]{figs/circ-simple.png}

    \includegraphics[height=1cm,align=c]{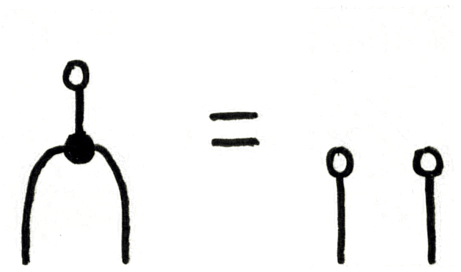}
    
    \includegraphics[height=1cm,align=c]{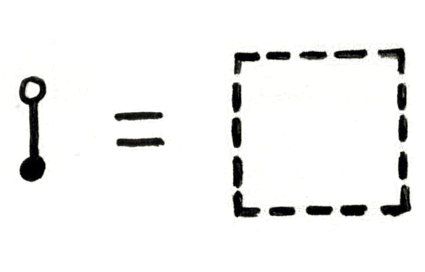}
  \end{mathpar}
  The associativity and unitality axioms become:
  \begin{mathpar}
    \includegraphics[height=1cm,align=c]{figs/associativity-axiom.png}
    
    \includegraphics[height=1cm,align=c]{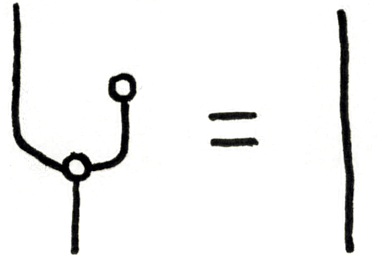}
    
    \includegraphics[height=1cm,align=c]{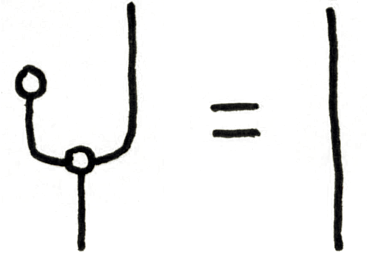}
  \end{mathpar}
  Now, define upside-down versions of the generators as in: 
  \begin{mathpar}
    \includegraphics[height=1cm,align=c]{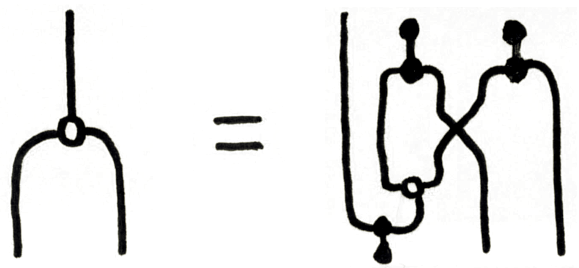}

    \includegraphics[height=1cm,align=c]{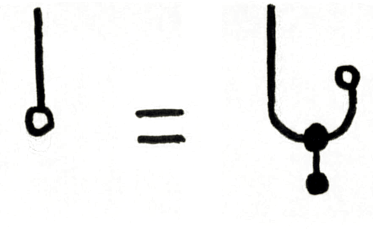}
  \end{mathpar}
  Then left cancellativity, right cancellativity, and conjugation are, respectively:
  \begin{mathpar}
  \includegraphics[height=1cm,align=c]{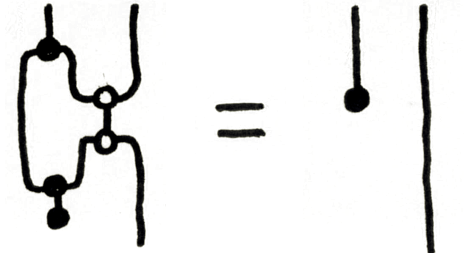}
    
  \includegraphics[height=1cm,align=c]{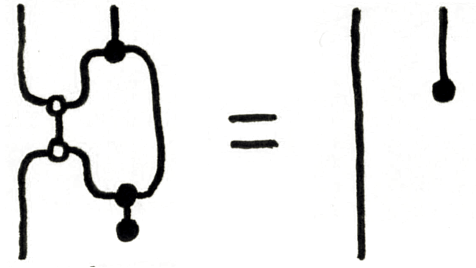}
  
  \includegraphics[height=1cm,align=c]{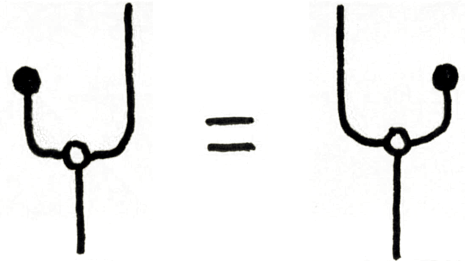}
  \end{mathpar}
  The corresponding category of models is the category of generalized separation algebras and partial monoid homomorphisms.
\end{example}


\begin{example}[Algebraic Theories]
  Let $\X$ be an algebraic theory, and let $(\X_{\mathsf{eq}})_\reglex$ be the regular completion of $\X$ ~\cite{Bunge1995,CarboniVitale1998}. $\Rel((\X_{\mathsf{eq}})_\reglex)$ is a relational algebraic theory. Further, its models and model morphisms (as a relational algebraic theory) coincide with the models and model morphisms of $\X$ (as an algebraic theory). Conversely, if $\X$ is a relational algebraic theory, then the maps of $\X$ form a subcategory $\Map(\X)$. $\Map(\X)$ has finite products, and so defines an algebraic theory in the usual sense. Further, the notions of model and model morphism for relational algebraic theories restrict to the usual notions for algebraic theories on the category of maps. 
\end{example}

\begin{example}[Essentially Algebraic Theories]
  An \emph{essentially algebraic theory} ~\cite{Palmgren2007} is (among many equivalent presentations) a category $\X$ with finite limits. Models are the finite-limit preserving functors $\X \to \SET$, and model morphisms are natural transformations. For $\X$ an essentially algebraic theory let $\X_\reglex$ be the regular completion of $\X$ ~\cite{CarboniVitale1998}. Then $\Rel(\X_\reglex)$ is a relational algebraic theory. Further, its models and model morphisms (as a relational algebraic theory) coincide with the models and model morphisms of $\X$ (as an essentially algebraic theory).
  Conversely, if $\X$ is a relational algebraic theory then the simple maps of $\X$ are a partial algebraic theory in the sense of ~\cite{partial-theories} -- which turn out to be equivalent to essentially algebraic theories. The notions of model and model morphism for relational theories restrict to the corresponding notions for partial theories. 
\end{example}


\section{The Variety Theorem}\label{sec:varietytheorem}

In this section we prove the variety theorem for relational algebraic theories. We do this in phases: first we introduce some necessary terminology concerning classes of idempotents, and recall some details of the idempotent splitting completion. Next, we make the relationship between bicategories of relations and regular categories precise. We then show how the situation extends to include exact categories, this being necessary because exactness is the difference between regular categories and definable categories. Finally, we introduce definable categories, which end up being the varieties of our relational theories. This is structured so that the variety theorem follows immediately. We end by showing precisely when two relational theories present the same definable category.

\subsection{Flavours of Idempotent Splitting}

We begin by introducing some important kinds of arrow in a relational theory:
\begin{definition}
  An arrow $f : A \to A$ of a relational algebraic theory is called reflexive in case $1 \leq f$, \emph{coreflexive} in case $f \leq 1$, a \emph{partial equivalence relation} in case it is symmetric and transitive as in:
    \begin{mathpar}
      \includegraphics[height=1cm,align=c]{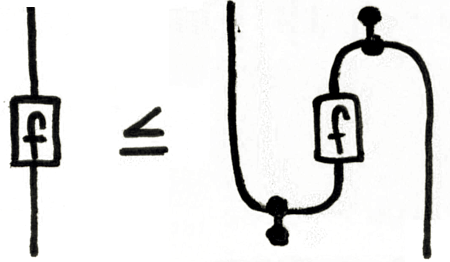}

      \includegraphics[height=1cm,align=c]{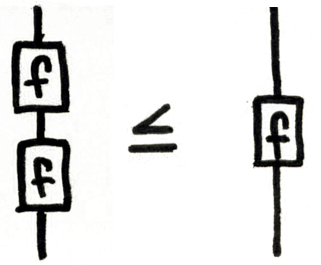}
    \end{mathpar}
  and is called an \emph{equivalence relation} if it is reflexive, symmetric, and transitive. 
\end{definition}
Notice in particular that every partial equivalence relation is idempotent, that  every coreflexive arrow is a partial equivalence relation, and that every equivalence relation is a partial equivalence relation. We also recall the idempotent splitting completion relative to a class of idempotents in a category:
\begin{definition}
  Let $\X$ be a category, and let $\mathcal{E}$ be a collection of idempotents in $\X$. Define a category $\Split_{\mathcal{E}}(\X)$ in which objects are pairs $(X,a)$ where $X$ is a object of $\X$ and $a : X \to X$ is in $\mathcal{E}$, and arrows $f : (X,a) \to (Y,b)$ are arrows $f : X \to Y$ of $\X$ such that $afb = f$. Composition is composition in $\X$, and identities are given by $a = 1_{(X,a)}: (X,a) \to (X,a)$.
\end{definition}
Every member of $\mathcal{E}$ splits in $\Split_\mathcal{E}(\X)$. It turns out that splitting partial equivalence relations works well with cartesian bicategories of relations:
\begin{proposition}[~\cite{Freyd90}]
  If $\X$ is a relational algebraic theory and $\mathcal{E}$ is a class of partial equivalence relations in $\X$, then $\Split_\mathcal{E}(\X)$ is a relational algebraic theory. 
\end{proposition}

\subsection{Tabulation and Regular Categories}
We begin our exposition of the correspondence between regular categories and relational algebraic theories by recalling the notion of tabulation ~\cite{CarboniWalters1987}. Intuitively, a tabulation of an arrow represents it as a subobject in the category of maps. 
\begin{definition}
  A \emph{tabulation} of an arrow $f : X \to Y$ in a relational algebraic theory $\X$ consists of a pair of maps $(h,k)$ such that the equation below on the left holds in $\X$, and the map below on the right is monic in $\Map(\X)$:
  \begin{mathpar}
    \includegraphics[height=1cm,align=c]{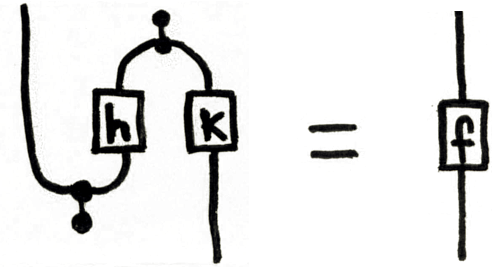}

    \includegraphics[height=1cm,align=c]{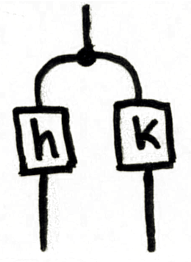}
  \end{mathpar}
  $\X$ is \emph{tabular} in case every arrow of $\X$ admits a tabulation. Further, define $\RAT_\tab$ to be the full 2-subcategory of $\RAT$ (Definition \ref{def:rat}) on the tabular 0-cells. 
\end{definition}
The category of maps of a tabular relational algebraic theory is regular, and conversely the category of internal relations in a regular category is tabular:
\begin{proposition}\label{prop:2fun}
  Let $\REG$ be the 2-category of regular categories, regular functors, and natural transformation. Then:
  \begin{enumerate}[(i)]
    \item If $\X$ is a tabular relational algebraic theory then $\Map(\X)$ is regular. This extends to a 2-functor $\Map : \RAT_{\tab} \to \REG$.
    \item If $\C$ is a regular category, then $\Rel(\C)$ is tabular. This extends to a 2-functor $\Rel : \REG \to \RAT_\tab$.
  \end{enumerate}
\end{proposition}
  
Tabular relational theories and regular categories are thus interchangeable:\footnote{We note that we restrict our attention to the 0- and 1-cells then this is proven in ~\cite{CarboniWalters1987}. Our contribution is to extend this to include 2-cells.}
\begin{theorem}\label{thm:tabregequiv}
  There is an equivalence of 2-categories $\Map : \RAT_\tab \simeq \REG : \Rel$.
\end{theorem}

Finally, any relational theory can be made tabular by splitting the coreflexives:
\begin{proposition}\label{thm:tabularcompletion}
  Let $\X$ be a relational algebraic theory, and let $\cor$ be the collection of coreflexives in $\X$. Then $\X$ is tabular if and only if every member of $\cor$ splits. In particular, $\Split_\cor(\X)$ is always tabular. This extends to a 2-adjunction $\Split_\cor : \RAT \dashv \RAT_\tab : U$ where $U$ is the evident forgetful functor.
\end{proposition}

\subsection{Effectivity and Exact Categories}

We begin by recalling the closely related notions of effectivity and exactness:
\begin{definition}[~\cite{Freyd90}] 
  A relational algebraic theory $\X$ is \emph{effective} in case all partial equivalence relations in $\X$ split. Let $\RAT_\eff$ be the full 2-subcategory of $\RAT$ on the effective 0-cells.
\end{definition}

\begin{definition}[~\cite{CarboniVitale1998}]
  A regular category $\C$ is \emph{exact} in case $\mathsf{Rel(\C)}$ is effective. Let $\EX$ be the full 2-subcategory of $\REG$ on the exact 0-cells. 
\end{definition}

It is straightforward to verify that Theorem \ref{thm:tabregequiv} restricts to the effective case:
\begin{proposition}
  If $\X$ is an effective relational algebraic theory, then $\Map(\X)$ is exact. Conversely, if $\C$ is an exact category, then $\Rel(\C)$ is effective. This extends to an equivalence of 2-categories $\Map : \RAT_\eff \simeq \EX : \Rel$.
\end{proposition}

Splitting equivalence relations makes tabular relational theories effective:
\begin{proposition}\label{thm:effectivecompletion}
  Let $\X$ be a tabular relational algebraic theory, and let $\eq$ be the collection of equivalence relations in $\X$. Then $\Split_\eq(\X)$ is effective. This extends to a 2-adjunction $\Split_\eq \RAT_\tab \dashv \RAT_\eff : U$ where $U$ is the evident forgetful functor.
\end{proposition}

We may therefore give the exact completion of a regular category as follows:
\begin{proposition}[~\cite{CarboniVitale1998,LackTendas20}]
  If $\C$ is regular, define the \emph{exact completion} of $\C$ by
  \[ \C_{\exreg} = \Map(\Split_{\eq}(\Rel(\X))) \]
  Then $\C_\exreg$ is exact. This extends to a 2-adjunction $\exreg : \REG \dashv \EX : U$ where $U$ is the evident forgetful functor.
\end{proposition}
We summarize the relationship of regularity and exactness to relational theories:
\begin{corollary}\label{cor:bigsquare}
  The following diagram of left 2-adjoint commutes: 
  \[
  \begin{tikzcd}
    \RAT_\tab \ar[r,"\Map","\sim"'] \ar[d,"\Split_\eq"'] & \REG \ar[d,"\exreg"] \\
    \RAT_\eff \ar[r,"\Map"',"\sim"] & \EX
  \end{tikzcd}
  \]
  where the arrows marked with $\sim$ are part of a 2-equivalence. 
\end{corollary}

Similarly, splitting partial equivalence relations allows us to summarize the role of the idempotent splitting completion:

\begin{proposition}\label{prop:splittingthing}
  Write $\per$ to denote the collection of partial equivalence relations in a relational algebraic theory. There is a 2-adjunction $\Split_\per : \RAT \dashv \RAT_\eff : U$ where $U$ is the evident forgetful functor.
  Further, for any relational algebraic theory $\X$, we have $\Split_\per(\X) \simeq \Split_\eq(\Split_\cor(\X))$, and so the following diagram of left 2-adjoints commutes:
    \[\begin{tikzcd}
    \RAT \ar[r,"\Split_\cor"] \ar[rd,"\Split_\per"'] & \RAT_\tab \ar[d,"\Split_\eq"] \\ & \RAT_\eff
    \end{tikzcd}\]
\end{proposition}
\begin{proof}
  The proof that $\Split_\per$ defines a 2-functor which is left adjoint to the forgetful 2-functor is straightforward, and similar to Proposition \ref{thm:tabularcompletion}. A proof that $\Split_\per(\X) \simeq \Split_\eq(\Split_\cor(\X))$ can be found in ~\cite[2.169]{Freyd90}, it follows immediately that our diagram of left 2-adjoints commutes. 
\end{proof}

\subsection{Definable Categories}

The final idea involved in our variety theorem is that of a definable category ~\cite{Kuber2018}. Definable categories come from categorical universal algebra. If we take regular categories as our notion of theory, regular functors into $\SET$ as our notion of model, and natural transformations as our model morphisms, then definable categories are the corresponding varieties. We follow the exposition of ~\cite{LackTendas20}, and in particular we formulate definable categories via finite injectivity classes:

\begin{definition}[Finite Injectivity Class]
  Let $h : A \to B$ be an arrow of $\X$. Then an object $C$ of $\X$ is said to be \emph{$h$-injective} in case the function of hom-sets $\X(h,C) : \X(B,C) \to \X(A,C)$ defined by $X(h,C)(f) = hf$ is injective. If $M$ is a finite set of arrows in $\X$, write $\mathsf{inj}(M)$ for the full subcategory on the objects $C$ of $\X$ that are $h$-injective for each $h \in M$. We say that each $\mathsf{inj}(M)$ is a \emph{finite injectivity class} in $\X$. 
\end{definition}
Definable categories are defined relative to an ambient locally finitely presentable category. It is an open problem to give a free-standing characterization ~\cite{Kuber2018}. 
\begin{definition}
  A category is said to be \emph{definable} if it arises as a finite injectivity class in some locally finitely presentable category. If $\X$ and $\Y$ are definable categories, a functor $F : \X \to \Y$ is called an \emph{interpretation} in case it preserves products and directed colimits. Let $\mathsf{DEF}$ be the 2-category with definable categories as 0-cells, interpretations as 1-cells, and natural transformations as 2-cells.
\end{definition}
From any definable category we can obtain an exact category by considering its interpretations into $\SET$.
\begin{proposition}[~\cite{LackTendas20}]
  If $\X$ is a definable category then the functor category $\DEF(\X,\SET)$ is an exact category. This extends to a 2-functor $\DEF(\_\,,\SET) : \DEF^\op \to \EX$. 
\end{proposition}
Similarly, for any regular category the associated category of regular functors into $\SET$ is definable. 
\begin{proposition}[~\cite{LackTendas20}]
  If $\C$ is a regular category then the functor category $\REG(\C,\SET)$ is definable. This extends to a 2-functor $\REG(\_\,,\SET) : \REG \to \DEF^\op$. 
\end{proposition}
If the category in question is exact, then considering interpretations of the resulting definable category into $\SET$ yields the original exact category. This lifts to the 2-categorical setting.
\begin{proposition}[~\cite{LackTendas20}]
  There is an adjunction of 2-categories $\REG(-,\SET) : \REG \dashv \DEF^\op : \DEF(-,\SET)$ which specializes to an equivalence of 2-categories $\REG(-,\SET) : \EX \simeq \DEF^\op : \DEF(-,\SET)$.
\end{proposition}
This gives another way to describe the exact completion of a regular category:
\begin{proposition}[~\cite{LackTendas20}]
  If $\C$ is regular then $\C_\exreg \simeq \DEF(\REG(\C\,,\SET)\,,\SET)$.
\end{proposition}
Thus, we may summarize the relationship between definable, regular, and exact categories as follows:
\begin{corollary}[{\cite[Section 9,10]{LackTendas20}}]\label{cor:deftriangle}
  The following diagram of left 2-adjoints commutes.
  \[\begin{tikzcd}
  \REG \ar[rd,"{\REG(-,\SET)}"] \ar[d,"\exreg"'] \\
  \EX \ar[r,"{\REG(-,\SET)}"',"\sim"] & \DEF^\op
  \end{tikzcd}\]
  where the arrow marked with $\sim$ is part of a 2-equivalence. 
\end{corollary}

The ingredients of our variety theorem for relational algebraic theories are now assembled. Together, Proposition \ref{prop:splittingthing}, Corollary \ref{cor:bigsquare}, and Corollary \ref{cor:deftriangle} give:
\begin{corollary}\label{cor:big-2-diagram}
  There following diagram of left 2-adjoints commutes:
\[\begin{tikzcd}
\RAT \ar[rd,"\Split_\per"'] \ar[r,"\Split_\cor"] & \RAT_\tab \ar[r,"\Map","\sim"'] \ar[d,"\Split_\eq"] & \REG \ar[d,"\exreg"'] \ar[dr,"{\REG(-,\SET)}"] \\
& \RAT_\eff \ar[r,"\Map"',"\sim"] & \EX \ar[r,"{\REG(-,\SET)}"',"\sim"] & \DEF^\op
\end{tikzcd}\]
where the arrows marked with $\sim$ are part of a 2-equivalence. 
\end{corollary}
Now our variety theorem is an immediate consequence of Corollary \ref{cor:big-2-diagram}:
\begin{theorem}\label{thm:relvariety}
  There is an adjunction of 2-categories $\mathsf{Mod} : \RAT \dashv \DEF^\op : \mathsf{Th}$
\end{theorem}

It may not be immediately clear what this tells us about the category of models and model morphisms of a relational algebraic theory, so let us briefly discuss. Consider an arbitrary relational algebraic theory $\X$. Our universe of models $\Rel$ is tabular, so models of $\X$ and models of $\Split_\cor(\X)$ are the same thing since the image of any coreflexive in $\X$ already splits in $\Rel$. Then the category of models of $\X$ and model morphisms thereof is $\RAT_\tab(\Split_\cor(\X),\Rel)$. When we transport this across the 2-equivalence $\Map : \RAT_\tab \stackrel{\sim}{\to} \REG$ it becomes $\REG(\Map(\Split_\cor(\X)),\SET)$, a definable category. Thus, categories of models and model morphisms of regular algebraic theories are definable categories. 

Now, $\SET$ is exact, so $\Rel$ is effective, which means that much like the models of $\X$ and $\Split_\cor(\X)$, the models of $\X$ and $\Split_\per(\X)$ are the same. We have shown that $\RAT_\eff \simeq \EX \simeq \DEF^\op$, and so the question of when two relational algebraic theories generate the same category of models and model morphisms can be answered as follows:
\begin{theorem}\label{thm:persplit}
  Two relational algebraic theories $\X$ and $\Y$ present equivalent definable categories if and only if $\Split_\per(\X)$ and $\Split_\per(\Y)$ are equivalent.
\end{theorem}
Compare this to the case of algebraic theories, in which two theories present the same variety in case splitting \emph{all} idempotents yields equivalent categories ~\cite{Adamek2003}. 

\bibliographystyle{plain}
\bibliography{citations}

\end{document}